\documentclass[12pt]{amsart}
\usepackage{amssymb}
\usepackage{amsmath}
\usepackage{latexsym}
\usepackage{amsthm}
\usepackage{rotate}
\usepackage{graphicx}
 \makeatletter \textwidth 6.0in \textheight
9.0in \oddsidemargin 0.0in \evensidemargin 0.0in \topmargin -0.2in
\def\thmhead@plain#1#2#3{%
 \thmname{#1}\thmnumber{\@ifnotempty{#1}{
 }#2}%
 \thmnote{ \the\thm@notefont(#3)}}
\let\thmhead\thmhead@plain
\def\swappedhead#1#2#3{%
 \thmnumber{#2}\thmname{\@ifnotempty{#2}{. }#1}%
 \thmnote{ \the\thm@notefont(#3)}}
\makeatother

\theoremstyle{definition} 
\newtheorem{definition}{Definition}[section]

\theoremstyle{plain}      
\newtheorem{proposition}[definition]{Proposition}
\newtheorem{theorem}[definition]{Theorem}

\newtheorem{lemma}[definition]{Lemma}

\begin{document}
\email{aa145@aub.edu.lb}

\title{Spans of Preference Functions for De Bruijn Sequences}
\author[A. Alhakim]{Abbas Alhakim Department of Mathematics\\ American University of Beirut\\ Beirut, Lebanon}
\begin{abstract}
A nonbinary Ford sequence is a de~Bruijn sequence generated by
simple rules that determine the priorities of what symbols are to be
tried first, given an initial word of size $n$ which is the order of
the sequence being generated. This set of rules is generalized by
the concept of a preference function of span $n-1$, which gives the
priorities of what symbols to appear after a substring of size $n-1$
is encountered. In this paper we characterize preference functions
that generate full de~Bruijn sequences. More significantly, We
establish that any preference function that generates a de~Bruijn
sequence of order $n$ also generates de~Bruijn sequences of all
orders higher than $n$, thus making the Ford sequence no special
case. Consequently, we define the preference function complexity of
a de Bruijn sequence to be the least possible span of a preference
function that generates this de~Bruijn sequence.
\end{abstract}

\maketitle

\section{Introduction}
Given a positive integer $t>1$ and an alphabet
$A=\{0,1,\ldots,t-1\}$ of size $t$, a de Bruijn sequence of order
$n$ over the alphabet $A$ is a sequence of symbols such that every
pattern of size $n$ appears exactly once as a block of contiguous
symbols. For example, $00110$ and $0011221020$ are two de~Bruijn
sequences of order $2$ over the alphabets $\{0,1\}$ and $\{0,1,2\}$
respectively. The existence of these sequences for any finite size
alphabet and any order is a well known fact \cite{dB46}.

For the binary alphabet, a classical but rather curious algorithm
that generates a de Bruijn sequence for any order $n$ is called the
``prefer-one'' algorithm. It consists of the following simple steps.
Begin by writing $n$ zeros. Then for $k>n$, write a one for the
$k^{th}$ bit of the sequence if the newly formed $n$-tuple has not
previously appeared in the sequence, otherwise write a zero. This is
repeated, preferring one every step of the way, until neither
appending  one nor zero puts a new $n$-tuple, at which time the
algorithm halts.

The prefer-one sequence is traced back to Martin \cite{Martin1934}.
But it has been rediscovered by many authors, see Fredricksen
\cite{Fred1982} for an exposition.

The prefer-one algorithm is generalized to an alphabet of size $t>2$
by preferring a higher value over a lower value. That is, once the
initial $n$ zeros are written, the value $t-1$ is appended if the
word formed by the $n$ most recent symbols is new, otherwise  $t-2$
is proposed, otherwise $t-3$, etc. This sequence was proposed by
Ford \cite{Ford1957} and it therefore bears his name. We will refer
to this generalization as the ``prefer-higher'' algorithm.

In this paper, we show that preferring higher values is not
necessary to obtain full de~Bruijn sequences. In fact, a binary
algorithm similar to the prefer-one was recently proposed in
\cite{Alhakim10}. This algorithm is called the prefer opposite as it
proposes a bit that is opposite to the bit most recently appended to
the sequence. Although the prefer opposite sequence is not a
de~Bruijn sequence, it only misses the constant word $1^n$. In the
non-binary case, other preferences can be constructed that yield
full sequences. For example, each diagram in
Table~\ref{Ta:Pref_Diagrams} can generate a full de~Bruijn sequence
of arbitrary order $n$ that starts with the initial word $0^n$. Each
row in a diagram displays the digits to be proposed, in decreasing
priority, when the rightmost digit of the sequence being constructed
is the digit that appears on the left side of the arrow of that row.
A proposed digit is accepted if the most recently formed word of
size $n$ has not appeared earlier in the sequence, otherwise the
next digit in that row is proposed.

It is worth noticing here that the upper and lower left diagrams
give the same decreasing preference regardless of the previous
digit. Thus they display the prefer-higher rules for alphabet sizes
$3$ and $4$ respectively.

The following are respectively all the sequences of order $2$ that
are generated using the diagrams in Table~\ref{Ta:Pref_Diagrams}.

$00221201100$; $00110221200$; $00120221100$;

$003323130221201100$; $001320221103312300$; $003020132233121100$;

In the sequel it will be proven that these diagrams generate
de~Bruijn sequences of all orders. More generally, we will
characterize all such diagrams that produce full sequences.

\begin{table}
\small
\begin{tabular}{c|c|c}
\begin{tabular}{lllll}
0 & $\rightarrow$ & 2, & 1, & 0\\
1 & $\rightarrow$ & 2, & 1, & 0\\
2 & $\rightarrow$ & 2, & 1, & 0\\
\end{tabular}
&
\begin{tabular}{lllll}
0 & $\rightarrow$ & 1, & 2, & 0\\
1 & $\rightarrow$ & 1, & 0, & 2\\
2 & $\rightarrow$ & 2, & 1, & 0\\
\end{tabular}
&
\begin{tabular}{lllll}
0 & $\rightarrow$ & 1, & 2, & 0\\
1 & $\rightarrow$ & 2, & 1, & 0\\
2 & $\rightarrow$ & 0, & 2, & 1\\
\end{tabular}\\\hline
\begin{tabular}{llllll}
0 & $\rightarrow$ & 3, & 2, & 1, & 0\\
1 & $\rightarrow$ & 3, & 2, & 1, & 0\\
2 & $\rightarrow$ & 3, & 2, & 1, & 0\\
3 & $\rightarrow$ & 3, & 2, & 1, & 0\\
\end{tabular}
&
\begin{tabular}{llllll}
0 & $\rightarrow$ & 1, & 2, & 3, & 0\\
1 & $\rightarrow$ & 3, & 1, & 0, & 2\\
2 & $\rightarrow$ & 0, & 2, & 1, & 3\\
3 & $\rightarrow$ & 2, & 3, & 1, & 0\\
\end{tabular}
&
\begin{tabular}{llllll}
0 & $\rightarrow$ & 3, & 2, & 1, & 0\\
1 & $\rightarrow$ & 3, & 2, & 1, & 0\\
2 & $\rightarrow$ & 0, & 2, & 3, & 1\\
3 & $\rightarrow$ & 0, & 2, & 3, & 1\\
\end{tabular}
\\\hline
\end{tabular}
\caption{Some preference diagrams of de~Bruijn sequences with
alphabet sizes 3 and 4.}\label{Ta:Pref_Diagrams}
\end{table}

\section{Main Results}

The idea of generating a de~Bruijn sequence by making preferences is
formalized in the concept of preference functions, defined in Golomb
\cite{Golomb1967} who attributes it to Welsh. In any de~Bruijn
sequence of order $n$, a word of size $(n-1)$ appears exactly $t$
times. A preference function gives the priority list of what digits
is to come first, second, third, etc. after a word of size $n-1$
appears in the sequence. Here is a precise definition.

\begin{definition}
A preference function  $P$ of span $n-1$ is a $t$-dimensional vector
valued function of $n-1$ variables such that, for each choice of the
vector ${\bf{a}}=(a_1,\ldots,a_{n-1})$ from the set $A^{n-1}$, the
entries of the vector
$\left(P_1({\bf{a}}),\ldots,P_t({\bf{a}})\right)$ form a permutation
of the elements of $A$.
\end{definition}

\begin{definition}
Given a preference function $P$, the least preference function
induced by $P$ is a function $\emph{g}$ from $A^{n-1}$ to $A^{n-1}$
defined as
$$g(a_1,\ldots,a_{n-1})=(a_2,\ldots,a_{n-1},P_t(a_1,\ldots,a_{n-1})).$$
\end{definition}

The following process is given in Golomb \cite{Golomb1967} and it
shows how a preference function of span $s-1$ is used to construct
recursive periodic sequences of order $s$.

\begin{definition}\label{D:procedure}
For any word $(I_1,\ldots,I_n)$ and preference function $P$ of span
$n-1$, the following inductive definition determines a unique finite
sequence $\{a_i\}$:
    \begin{enumerate}
    \item[1.] $a_1=I_1,\ldots,a_n=I_n$.
    \item[2.] If $a_{N+1},\ldots,a_{N+n-1}$ have been defined, then
    $a_{N+n}=P_i(a_{N+1},\ldots,a_{N+n-1})$, where $i$ is the
    smallest integer such that the word
    $$(a_{N+1},\ldots,a_{N+n-1},P_i(a_{N+1},\ldots,a_{N+n-1})$$ has
    not previously appeared as a segment of the sequence (provided
    that there is such $i$).
    \item[3.] Let $L=L\{a_i\}$ be the first value of $N$ such that no
    $i$ can be found to satisfy item 2. Then $a_{L+n-1}$ is the last
    digit of the sequence and $L$ is called the cycle period.
    \end{enumerate}
\end{definition}

Conversely, we remark that any periodic sequence induces at least
one preference function whose corresponding sequence is the periodic
sequence itself. To see this, consider a periodic sequence $S$
started at the word $J_1,\ldots,J_n$, where $n$ is the smallest word
size such that every pattern of size $n$ occurs at most once in $S$.
Now consider all occurrences (if any) of a pattern $\bf{w}$ of size
$n-1$ in a single period of $S$. Since every pattern of size $n$
occurs at most once, the number of occurrences $r(\bf{w})$ of the
pattern $\bf{w}$ is bounded above by $t$. For $i=1$ to $r(\bf{w})$
let $P_i(\bf{w})$ be the digit that occurs right after the $i^{th}$
occurrence of $\bf{w}$. If $r(\bf{w})<t$ let
$P_{r(\bf{w})+1}(\bf{w}),\ldots, P_t(\bf{w})$ be any permutation of
the digits which do not appear as entries of
$(P_1(\bf{w}),\ldots,P_{r(\bf{w})})$.

By the above construction, it is evident that the preference
function $P$ along with the initial word $J_1,\ldots, J_n$ produces
the sequence $S$. The next proposition follows immediately by the
above discussion.

\begin{proposition}\label{P:Prefunc_dBEquiv}
Fixing an initial word $(I_1,\ldots, I_n)$, there is a one to one
correspondence between the set of de Bruijn sequences of order $n$
and the set of preference functions of span $n-1$ which generate
de~Bruijn sequences of order $n$ started at $(I_1,\ldots, I_n)$.
\end{proposition}


Given an arbitrary preference function, a natural question is
whether or not this preference function generates a full de~Bruijn
sequence. In this section, we take on the problem of characterizing
such \textit{complete} preference functions.

For completeness, we now state two theorems, given in Golomb
\cite{Golomb1967}, which present conditions on a preference function
to produce a de~Bruijn sequence.

\begin{definition} For $0\leq r\leq n-1$, we say that $(x_1,\ldots,x_{n-1})$
has an $r$-overlap with $(I_1,\ldots,I_n)$ if
$(x_{n-r},\ldots,x_{n-1})=(I_1,\ldots,I_r)$. Notice that any
$(n-1)$-digit word at least has a zero overlap with
$(I_1,\ldots,I_n)$.
\end{definition}

\textbf{Theorem 16.} (of Golomb's Chapter VI) For any initial word
$(I_1,\ldots,I_{n})$, if $P$ is a preference function of span $n-1$
that satisfies $P_t(x_1,\ldots,x_{n-1})=I_{r+1}$, when $r$ is the
largest integer such that $(x_1,\ldots,x_{n-1})$ has an $r$-overlap
with $(I_1,\ldots,I_{n})$, then the sequence generated by
$(I_1,\ldots,I_{n})$ and $P$ has length $t^n$, i.e., it is a
de~Bruijn sequence of order $n$.

While the previous theorem states a condition that guarantees that a
preference function of span $n-1$ produces a de Bruijn sequence of
order $n$, the next theorem starts with a preference function that
is known to generate a de~Bruijn sequence of order $n-1$ and
provides a way to construct a preference function of span $n-1$ that
produces a de~Bruijn sequence f order $n$. This \textit{recursive
construction} is stated and proved for the binary case in Golomb
\cite{Golomb1967}, although it is claimed that the theorem can be
easily generalized to the non-binary case.

\textbf{Theorem 17.} (of Golomb's Chapter VI) The following
hypotheses are adopted:

1. Let $(I_1,\ldots, I_n)$ be an arbitrary initial word.

2. Let $P(x_1,\ldots,x_{n-2})=(P_1,P_2)$ be the preference function
for the binary de~Bruijn sequence of order $n-1$, $\{b_i\}$, and
initial word $(I_1,\ldots,I_{n-1})$, such that $P_1(I_1,\ldots,
I_{n-2})=1+I_{n-1} \mod 2$.

3. Let $x_1\oplus F(x_2,\ldots,x_{n-1})$ be the feedback formula for
$\{b_i\}$. That is, $b_{i} = b_{i-n+1}\oplus
F(x_{i-n+2},\ldots,x_{i-1})$ for all $i$

4. Let $P^{*}(x_1,\ldots, x_{n-1})=(P_1^*,P_2^*)$ which satisfies
$P_2^*=1\oplus P_1^*$ and
$$0=[P_1^*(x_1,\ldots,x_{n-1})\oplus1\oplus x_1\oplus F(x_2,\ldots,x_{n-1})]\times%
[x_1\oplus F(x_2,\ldots,x_{n-1})\oplus P_1(x_2,\ldots,x_{n-1})].$$

where the $\oplus$ is taken as addition modulo $2$. It follows from
these hypotheses that the sequence $\{a_i\}$, generated by
$(I_1,\ldots,I_n)$ and $P^*$, is a de~Bruijn sequence of order $n$.

While the conditions stated in Theorem~17 indeed generate a complete
binary preference table, these conditions are in a sense
artificially designed to make possible the inductive proof given in
Golomb \cite{Golomb1967}. In what follows we will show that, for any
alphabet size $t$, a preference function of span $n-1$ is itself
capable of generating de~Bruijn sequences of all orders larger than
or equal to $n$. Before we do this we  will characterize preference
functions of span $n-1$ that generate de~Bruijn sequences of order
$n$, i.e., complete preference functions.

A de~Bruijn sequence of order $n$ can be started with any of its
words of size $n$. Unless otherwise stated, in the rest of this
paper we will only be concerned with an initial word
$(I_1,\ldots,I_n)=0^n$, i.e. the constant string of $n$ zeros.

\begin{definition} Let $E$ be a finite set, let $f$ be a function from $E$ to
itself and let $l\geq1$ be an integer. By a cycle of length $l$
induced by $f$ we mean a sequence of elements $x_1,\ldots,x_l$ such
that $f(x_i)=x_{i+1}$ for $i=1$ to $l-1$ and $f(x_l)=x_1$.
\end{definition}

We now state our first main result.

\begin{theorem}\label{T:NoCycles}
Let $P$ be a complete preference function of span $n-1$ that
corresponds to a de~Bruijn sequence started at the string $0^n$.
Then the least preference function $g(x_1,\ldots,x_{n-1})$ has no
cycles of any length except the self-loop $(0^{n-1},0^{n-1})$, i.e.
$g(0^{n-1})=0^{n-1}$, which must be a cycle of $g$.
\end{theorem}

\begin{proof}
Let $S$ be the de Bruijn sequence starting with $0^n$ and resulting
from $P$. If $(0^{n-1},0^{n-1})$ is not a cycle of $g$ then
$P_t(0^{n-1})=a\neq0$ and hence $0$ has a higher preference over
$a$, that is $P_i(0^{n-1})=0$ for some $i<t$. Since $S$ is a
de~Bruijn sequence of order $n$, the word $0^{n-1}a$ must be a
substring. This means that the word $(0^{n-1},P_i(0^{n-1})=0^n$ must
have been proposed and accepted earlier in the sequence. So that
$0^n$ occurs twice in the sequence, which is a contradiction.
Suppose now that $g$ has a cycle of length $i$, $1\leq i\leq
t^{n-1}-1$ other than the self loop at $0^{n-1}$. Namely, suppose
that for some $y_1,\ldots,y_{i+n-1}$
\begin{eqnarray}\label{D:cycle}
g(y_1,\ldots,y_{n-1})&=&(y_2,\ldots,y_{n})\\
g(y_2,\ldots,y_{n})&=&  (y_3,\ldots,y_{n+1})\notag\\\notag
                   &\vdots&\\\notag
g(y_{i},\ldots,y_{i+n-2})&=& (y_{i+1},\ldots,y_{i+n-1}).
\end{eqnarray}
where $(y_{i+1},\ldots,y_{i+n-1})=(y_1,\ldots,y_{n-1})$ but
$(y_j,\ldots,y_{j+n-2})\neq (y_k,\ldots,y_{k+n-2})$ for all pairs
(j,k) such that $1\leq j<k\leq i+1$ and $(j,k)\neq(1,i+1)$.

Since $S$ is a de Bruijn sequence of order $n$, $(y_1,\ldots, y_n)$
occurs in $S$. By definition of $g$ and the first equation in
Display~(\ref{D:cycle}), $y_n=P_t(y_1,\ldots,y_{n-1})$. It follows
that all the words $(y_1,\ldots,y_{n-1},z)$, $z\neq y_n$ must have
occurred earlier in the sequence. This implies that all the
predecessors $(y,y_1,\ldots,y_{n-1})$, for $y\in A$ have occurred
before $(y_1,\ldots,y_n)$. In particular
$(y_i,\ldots,y_{i+n-1})=(y_i,y_1,\ldots,y_{n-1})$ has occurred
before $(y_1,\ldots,y_n)$.

Now, $g(y_i,\ldots,y_{i+n-2})=(y_{i+1},\ldots,y_{i+n-1})$ is
equivalent to $P_t(y_i,\ldots,y_{i+n-2})=y_{i+n-1}$. Using the above
argument, we see that all the words $(y_i,\ldots,y_{i+n-2},z)$,
$z\neq y_{i+n-1}$, and therefore all their predecessors
$(y,y_i,\ldots,y_{i+n-2})$ must have occurred earlier in the
sequence. In particular, $(y_{i-1},y_i,\ldots,y_{i+n-2})$ occurs
before $(y_i,\ldots,y_{i+n-1})$, which was just shown to occur
before $(y_1,\ldots,y_n)$. Repeating the same reasoning a total of
$i$ times, we see that $(y_1,\ldots,y_n)$ must have occurred earlier
in the sequence $S$. That is, $(y_1,\ldots,y_{n})$ occurs more than
once in $S$, contradicting the assumption that the latter is a
de~Bruijn sequence. This establishes that $g$ has no cycles besides
$(0^{n-1},0^{n-1})$.
\end{proof}

It is important to remark here that in the above theorem, the
initial word must be $0^n$ or--of course--a constant string $i^{n}$
(in which case the self loop of $g$ must be $(i^{n-1},i^{n-1})$).
Indeed, Table~\ref{Ta:counterexLOOP} displays a de~Bruijn sequence
of alphabet size $3$ and order $3$ with its corresponding preference
function. Noting that the initial word is $001$, we can see that the
induced least preference function $g$ has the cycle $(00,01,10,00)$.

\begin{table}[h]
\small
\begin{tabular}{lcl|lcl|lcl}
00 & $\rightarrow$ & 0,2,1 & 10 & $\rightarrow$ & 1,2,0 & 20 &
$\rightarrow$ & 1,0,2\\
01 & $\rightarrow$ & 1,2,0 & 11 & $\rightarrow$ & 0,1,2 & 21 &
$\rightarrow$ & 2,1,0\\
02 & $\rightarrow$ & 0,1,2 & 12 & $\rightarrow$ & 1,2,0 & 22 &
$\rightarrow$ & 2,0,1\\
\hline
\end{tabular}
\caption{Preference rules for the sequence
$00110121222010200021112022100$.}\label{Ta:counterexLOOP}
\end{table}\label{T:counterexLOOP}



The converse of Theorem~\ref{T:NoCycles} is also true. That is, if a
given preference function of span $n-1$ induces a least preference
function $g$ that has no cycles except the self loop at $0^{n-1}$,
then the preference function produces a de~Bruijn cycle of order $n$
started at the word $0^n$. However, the next result is much stronger
than this converse. We state it after the following algorithm.

\textbf{Algorithm P}

Input: Two integers $s>1$ and $n\geq s$ and a preference function
$P$ of span $s-1$.

Output: a unique de~Bruijn sequence $S=\{a_i\}$ of order $n$.
\begin{enumerate}
    \item[1.] $a_1=0,\ldots, a_n=0$.
    \item[2.] If $a_{N+1},\ldots,a_{N+n-1}$ have been defined, then
    $a_{N+n}=P_i(a_{N+n-s+1},\ldots,a_{N+n-1})$, where $i$ is the
    smallest integer between $1$ and $t$ such that the word
    $$(a_{N+1},\ldots,a_{N+n-1},P_i(a_{N+n-s+1},\ldots,a_{N+n-1})$$ has
    not previously appeared as a segment of the sequence (provided
    that there are such $i$).
    \item[3.] Let $L=L\{a_i\}$ be the smallest value of $N$ such that no
    $i$ can be found to satisfy the condition in (2). Then $a_{L+n-1}$ is the last
    digit of the sequence and $L$ is called the cycle period.
    \end{enumerate}

\begin{theorem}\label{T:main_converse}
Let $P$ be a preference function of span $s-1$ that induces a least
preference function $g$ which admits no cycles except the self loop
$(0^{s-1},0^{s-1})$. Then for any integer $n\geq s$ the sequence
given by Algorithm~P is a de~Bruijn sequence of order $n$.
\end{theorem}

We observe that this theorem establishes that the Ford sequence is
rather the norm than the exception. For the Ford sequence, the
permutation $(t-1,t-2,\ldots,0)$--which is a de~Bruijn sequence of
order $1$--generates de~Bruijn sequences of all orders. Using
Theorem~\ref{T:main_converse}, given any de~Bruijn sequence of order
$s$, we can construct the corresponding preference function of span
$s-1$ which in turn can generate a unique de~Bruijn sequence of any
order higher than $s$. The proof of this theorem will be given after
a few lemmas are formulated and proved. %

\begin{lemma}\label{L:EndOfSequence}
The sequence $S$ in Theorem~\ref{T:main_converse} ends just after
the word $a0^{n-1}$ is encountered, for some $a\in A$, $a\neq0$.
\end{lemma}
\begin{proof}
First, it is immediate by the construction in \textbf{Algorithm P}
that a word of size $n$ occurs at most once in the constructed
sequence. Suppose now that the algorithm terminates just after the
word $(x_1,\ldots,x_n)\neq a0^{n-1}$ is realized. That is,
$(x_2,\ldots,x_n,y)$ must have appeared earlier in the sequence for
all $y\in A$. This implies that $(x_2,\ldots,x_n)$ appeared $t+1$
times. Since $(x_2,\ldots,x_n)$ is not equal to $0^{n-1}$, it is not
the initial block of the sequence so that every time it appeared it
was preceded by something. The pigeon hole principle thus implies
that there exists an element $z\in A$ such that $(z,x_2,\ldots,x_n)$
occurs twice in the sequence, which is a contradiction.
\end{proof}

\begin{lemma}\label{L:AllZeroPredecessors}
All words of the form $b0^{n-1}$ occur in the sequence $S$ of
Theorem~\ref{T:main_converse}.
\end{lemma}
\begin{proof}
By Lemma~\ref{L:EndOfSequence} the sequence ends with the word
$a0^{n-1}$. The word $0^{n}=0^{n-1}0$, which already occurs in the
beginning, can not be appended after $a0^{n-1}$. Since
$P_t(0^{n-1})=0$, no other symbol $z$ can be appended either. This
implies that all words of the form $0^{n-1}z$ have occurred earlier
in the sequence. It follows that all the words of the form
$b0^{n-1}$, where $b\neq0$  must occur in the sequence.
\end{proof}

\begin{lemma}\label{L:NextNotAppear}
If $X_1=(x_1,\ldots,x_n)$ is a word that does not occur in $S$ then
neither does the word $X_2=(x_2,\ldots,x_n,c)$, where
$c=P_t(x_{n-s+2},\ldots,x_n)$.
\end{lemma}
\begin{proof}
Suppose that $X_2$ occurs in $S$. $X_2$ can not be the zero string
$0^n$ because the latter occurs as the first string. Hence $X_2$ is
preceded by some string. Since $c=P_t(x_{n-s+2},\ldots,x_n)$ has the
least preference, it follows that all the words $(x_2,\ldots,x_n,z)$
must have occurred earlier in the sequence. Therefore the set of all
predecessors $(y,x_2,\ldots,x_n)$ must have occurred for all values
of $y$. In particular, $X_1=(x_1,\ldots,x_n)$ must have occurred,
which is a contradiction.
\end{proof}

\begin{proof} (of Theorem~\ref{T:main_converse})
Suppose that $(x_1,\ldots,x_n)$ is a pattern that does not appear in
$S$.

Case 1. Let us first suppose that $(x_{n-s+2},\ldots,x_n)=0^{s-1}$.
Since $(x_1,\ldots,x_n)$ can not be all zeros, there must exist an
integer $i$, $1\leq i\leq n-s+1$ such that $x_i\neq0$ but $x_j=0$
for all $j$, $i+1\leq j\leq n$. Since $g(0^{s-1})=0^{s-1}$ it is
clear that $P_t(0^{s-1})=0$. It follows by applying
Lemma~\ref{L:NextNotAppear} that $(x_2,\ldots,x_i,0^{n-i+1})$ does
not occur in $S$. By the same argument, applying
Lemma~\ref{L:NextNotAppear} another $(i-2)$ times, we see that the
word $x_i,0^{n-1}$ does not occur. This contradicts
Lemma~\ref{L:AllZeroPredecessors}.

Case 2. Suppose now that $(x_{n-s+2},\ldots,x_n)\neq0^{s-1}$. Then
by Lemma~\ref{L:NextNotAppear} the word $(x_{2},\ldots,x_n,x_{n+1})$
does not appear either, where $x_{n+1}=P_t(x_{n-s+2},\ldots,x_n)$.
Moreover, $(x_{n-s+2},\ldots,x_n)\neq (x_{n-s+3},\ldots,x_{n+1})$
for otherwise the least preference function $g$ would have a cycle
of length $1$ that is distinct from the self loop
$(0^{s-1},0^{s-1})$, namely
\[
(x_{n-s+1},\ldots,x_n, g(x_{n-s+1},\ldots,x_n))=
(x_{n-s+2},\ldots,x_{n+1}),
\]
which can not be the case by the given.

If $(x_{n-s+3},\ldots,x_{n+1})=0^{s-1}$, Case~1 above leads to a
contradiction. So it is safe to assume that this is not the case. We
claim that, by applying this argument repeatedly, we  eventually get
a word ending with $0^{s-1}$ that does not occur in $S$. To see this
note that, after $i$ repetitions of
Lemma~\ref{L:NextNotAppear}--with $i\leq t^{s-1}$, we conclude that
the word $(x_{n+i-s+2},\ldots, x_{n+i})$ does not occur in $S$,
where for $j=n+1$ to $n+i$, $x_j=P_t(x_{j-s+1},\ldots,x_{j-1})$ and
$(x_{j-s+2},\ldots, x_j)\neq0^{s-1}$. Since
$g(x_{j-s+1},\ldots,x_{j-1})=(x_{j-s+2},\ldots,x_j)$ and since $g$
has no cycles of any length (namely, no cycles of length
$1,2,\ldots,t^{s-1}-1$) other than the self loop at $0^{s-1}$, we
see that $(x_{j-s+1},\ldots,x_j)\neq
(x_{j^{\prime}-s+2},\ldots,x_{j^{\prime}})$ for all $j<j^{\prime}$
and $n\leq j,j^{\prime}\leq n+i$. Otherwise, the sequence
\[
(x_{j-s+1},\ldots,x_j),(x_{j-s+2},\ldots,x_{j+1}),
\cdots,(x_{j^{\prime}-s+1},\ldots,x_{j^{\prime}})
\]
would form a cycle of length $j^{\prime}-j$.

For any $i$ such that $1\leq i\leq t^{s-1}-2$, if the right tail
$(x_{n+i-s+2},\ldots,x_{n+i})=0^{s-1}$ then applying Case~1 leads to
a contradiction. Suppose then that the right tail is distinct from
$0^{s-1}$ for all $i=1$ to $t^{s-1}-2$.  Then, for $i=t^{s-1}-1$,
the facts that all the words are distinct and that there are
$i+1=t^{s-1}$ words imply that the last word of size $s-1$ is
necessarily equal to $0^{s-1}$, thus leading to a contradiction, by
Case~1. This establishes the theorem.
\end{proof}

\section{Preference Function Complexity}

In the vast literature on de~Bruijn sequences, there has been more
than one method to classify these sequences. One well known
criterion for binary de~Bruijn sequences is the number of ones in
the truth table of the corresponding feedback function, (namely, the
function $F$ defined in the statement of item (3) in Theorem~17
above).

Also, de~Bruijn sequences have been classified according to their
linear complexity, which is defined as the minimal span of a linear
shift register that generates the de~Bruijn sequence. In other
words, it is the minimal integer $N$ such that there exists a linear
feedback function $F=F(x_2,\ldots,x_N)$ that can generate the
de~Bruijn sequence.

It was proven by Chan, Games and Key \cite{ChanGamesKey1982} that
the linear complexity of a binary de Bruijn sequence of order $n$ is
between $2^{n-1}+n$ and $2^n-1$.

In this section,  we use Theorem~\ref{T:main_converse} to introduce
a new notion of complexity of de~Bruijn sequences  of any alphabet
size that relates to the preference function which generates the
sequence. We thus obtain another classification of de~Bruijn
sequences based on this complexity.

To fix ideas, we observe that, by Theorem~\ref{T:main_converse}, it
is clear that an algorithm such as the one that generates the Ford
sequence (for general alphabet size $t$) is rather the norm than the
exception. For the latter algorithm, a preference function that
generates a de Bruijn sequence of order $1$ also generates a
de~Bruijn sequence of any order $n$ larger than $1$ when started
with the initial word $0^n$.

Let us also observe that the preference function of the Ford
sequence--with the ``prefer-higher'' algorithm--is a constant
function of span $0$ that is given by $P(x)=(t-1,t-2,\ldots,0)$ for
all $x$ in $A$. See Table~\ref{Ta:Pref_Diagrams}. The corresponding
least preference function is given by $g(x)\equiv0$, which admits
the only cycle $g(0)=0$. In fact, the two preference diagrams in the
leftmost column of Table~\ref{Ta:Pref_Diagrams} have span zero while
the remaining diagrams have span one.

\begin{definition}
Given an order $n$ de~Bruijn sequence $S$ that starts with the fixed
word $\textbf{0}^n$ we define the preference function complexity
$comp_{\textbf{0}}(S)$ as the smallest integer $s$, $0\leq s\leq n$
such that there exists a preference function of span $s$ that
generates the sequence $S$ with the initial word $\textbf{0}^n$.
\end{definition}

The sixteen binary de~Bruijn sequences of order four are given in
Table~\ref{Ta:dB4Sequences} while their corresponding preference
functions are given in Table~\ref{t:prefdB4Seq}. Notice that
sequence (3)--the Ford sequence--does not depend on any of the
previous three bits so it has preference function span $0$ while
sequence (6) depends only on the previous two bits so it has span 2.
all other sequences have preference function span 3. Thus they have
full span. There are no sequences with span 1, due to the binary
alphabet. Sequence (2) comes close. In fact, changing the preference
of '111' to 1 then 0 makes the preference function depend only on
the previous bit but this introduces a self loop $111\rightarrow111$
in the corresponding least preference function so the resulting
sequence misses the word $1111$. Note that this sequence is the
prefer opposite sequence mentioned earlier.

\begin{proposition}
The distribution of de~Bruijn sequences of order $n$ according to
their preference function complexity is given by $N_0(n)=(t-1)!$,
$N_1(n)=((t-1)!)^t\cdot t^{t-2}$, and for $i>1$
$N_i(n)=((t-1)!)^{t^i}\cdot t^{t^i-i-1}- ((t-1)!)^{t^{i-1}}\cdot
t^{t^{i-1}-i}$ where, for $i=0$ to $n-1$,
$$N_i(n)=card\{S: S \textup{ is a
de Bruijn sequence of order n such that } comp(S)=i\}.$$
\end{proposition}

\begin{proof}
For $i=0$ the order of preference does not depend on any of the
previous digits, in particular it does not depend on the immediately
previous digit. Since the only allowed cycle in the induced least
preference function $g$ of  is the self loop from $0$ to $0$ it
follows that $g(i)=0$ for all digits $i$. The remaining $t-1$ digits
can be given any of $(t-1)!$ orders of preference.

For $i\geq1$, it is evident that the complexity of a de~Bruijn
sequence of order $i$ does not exceed $i-1$. Moreover, it is well
known, see \cite{FlieSainteMarie1894}, that the total number of
de~Bruijn sequences of order $i$ is given by the formula
$M(t,i)=[(t-1)!]^{t^{i-1}}\cdot t^{t^{i-1}-i}$. Hence $N_1(1)$ is
$M(t,1)$ minus the number of sequences of complexity $0$.

Similarly, $N_i(i)$ is $M(t,i)$ minus the number of sequences whose
complexity is less than $i$. Since Theorem~\ref{T:main_converse}
implies that every preference function of span $i<n$ also produces a
de~Bruijn sequence of order $n$, it follows that $N_i(n)=N_i(i)$.
\end{proof}

We will say that two sequences $\{a_i\}$ and $\{b_i\}$ are
equivalent if $b_i=\sigma(a_i)$ for some permutation $\sigma$ of the
alphabet $A$. Our last result relates to de~Bruijn sequences with
complexity zero. Notice that while the binary case allows only one
preference function with zero span, higher values of $t$ yield
$(t-1)!$ cases. The following proposition shows that in fact all of
these cases yield equivalent de~Bruijn sequences.

\begin{proposition}
All de Bruijn sequences of preference function complexity zero are
equivalent, up to a permutation of the digits, to the Ford sequence.
\end{proposition}

\begin{proof}
Let $Q$ be an arbitrary complete preference function of span zero.
Evidently, there exists a permutation $\sigma$ such that
$Q_i=(\sigma(t-1),\ldots,\sigma(0))$ for all $i\in A$. Let $\{b_i\}$
be the sequence of order $n$---started at $0^n$---that corresponds
to $Q$. Consider now the sequence $\{a_i\}$ defined by
$a_i=\sigma^{-1}(b_i)$, which is obviously a de~Bruijn sequence. We
claim that $\{a_i\}$ is the Ford sequence of order $n$. To see this,
let $i_1<i_2<\ldots<i_t$ be the positions of a pattern
$x_1,\ldots,x_{n-2}$ in the sequence $\{a_i\}$. It follows that
$i_1,\ldots,i_t$ are the positions of the pattern
$\sigma(x_1),\ldots,\sigma(x_{n-2})$ in the sequence $\{b_i\}$. By
definition of $Q$, the substrings
$$(b_{i_1},\ldots,b_{i_1+n-2}),\ldots,(b_{i_t},\ldots,b_{i_t+n-2})$$
are followed respectively by $\sigma(t-1),\ldots, \sigma(0)$.
Therefore, the substrings
$$(a_{i_1},\ldots,a_{i_1+n-2}),\ldots,(a_{i_t},\ldots,a_{i_t+n-2})$$
of $\{a_i\}$ are followed by $t-1,t-2,\ldots,0$.

Since this is true for any pattern $(x_1,\ldots,x_{n-1})$, the proof
is complete.
\end{proof}

\begin{table}
\small
\begin{tabular}{l|l||l|l}
\hline
1 &  0000100110101111000 & 9 &  0000101111001101000\\
2 &  0000101001101111000 & 10 & 0000101111010011000\\
3 &  0000111101100101000 & 11 & 0000101100111101000\\
4 &  0000111101011001000 & 12 & 0000110010111101000\\
5 &  0000100111101011000 & 13 & 0000111101001011000\\
6 &  0000101001111011000 & 14 & 0000110100101111000\\
7 &  0000110111100101000 & 15 & 0000101101001111000\\
8 &  0000110101111001000 & 16 & 0000111100101101000\\
\hline
\end{tabular}
\caption{Binary de~Bruijn sequences of order
4.}\label{Ta:dB4Sequences}
\end{table}

\begin{table}
\small
\begin{tabular}{l|l|l|l|l|l|l|l|l|l|l|l|l|l|l|l|l}
\hline
     & \multicolumn{16}{c}{sequence}\\
\cline{2-17}
     & 1 & 2 & 3 & 4 &  5 & 6 & 7 & 8 & 9 & 10 & 11 & 12 & 13 & 14 & 15 & 16\\
\hline%
000 & 1,0 &  1,0 & 1,0 & 1,0 & 1,0 & 1,0 & 1,0 & 1,0
& 1,0 & 1,0 & 1,0 & 1,0 & 1,0 & 1,0 & 1,0 & 1,0\\%
001 & 0,1 & 0,1 & 1,0 & 1,0 & 0,1 & 0,1 & 1,0 & 1,0
& 0,1 & 0,1 & 0,1 & 1,0 & 1,0 & 1,0 & 0,1 & 1,0\\%
010 & 0,1 & 1,0 & 1,0 & 1,0 & 0,1 & 1,0 & 1,0 & 1,0
& 1,0 & 1,0 & 1,0 & 1,0 & 0,1 & 0,1 & 1,0 & 1,0\\%
011 & 0,1 & 0,1 & 1,0 & 1,0 & 1,0 & 1,0 & 0,1 & 1,0
& 1,0 & 1,0 & 0,1 & 0,1 & 1,0 & 0,1 & 0,1 & 1,0\\%
100 & 1,0 &  1,0 & 1,0 & 1,0 & 1,0 & 1,0 & 1,0 & 1,0 & 1,0 & 1,0 &
1,0 & 1,0 & 1,0 & 1,0 & 1,0 & 1,0\\
101 & 0,1 & 0,1 & 1,0 & 0,1 & 0,1 & 0,1 & 1,0 & 0,1
& 1,0 & 1,0 & 1,0 & 1,0 & 0,1 & 0,1 & 1,0 & 1,0\\%
110 & 1,0 & 1,0 & 1,0 & 1,0 & 1,0 & 1,0 & 0,1 & 1,0
& 0,1 & 1,0 & 0,1 & 0,1 & 1,0 & 1,0 & 1,0 & 0,1\\
111 & 1,0 &  1,0 & 1,0 & 1,0 & 1,0 & 1,0 & 1,0 & 1,0 & 1,0 & 1,0 & 1
,0 & 1,0 & 1,0 & 1,0 & 1,0 & 1,0\\\hline%

span  & \multicolumn{1}{c|}{3} & \multicolumn{1}{c|}{3} &
\multicolumn{1}{c|}{0} & \multicolumn{1}{c|}{3} &
\multicolumn{1}{c|}{3} & \multicolumn{1}{c|}{2} &
\multicolumn{1}{c|}{3} & \multicolumn{1}{c|}{3} &
\multicolumn{1}{c|}{3} & \multicolumn{1}{c|}{3} &
\multicolumn{1}{c|}{3} & \multicolumn{1}{c|}{3} &
\multicolumn{1}{c|}{3} & \multicolumn{1}{c|}{3} &
\multicolumn{1}{c|}{3} & \multicolumn{1}{c}{3}\\
\hline
\end{tabular}
\caption{preference functions of order 4 binary de~Bruijn sequences
and their spans.}\label{t:prefdB4Seq}
\end{table}




\end{document}